\documentclass[12pt]{amsart}

\usepackage{amssymb,amsthm,amscd,wasysym,mathrsfs,amsxtra,mathtools}
\usepackage{amsfonts}
\input xy
\xyoption{all}

\usepackage{newtxtext}
\usepackage{newtxmath}
\usepackage[cal=cm]{mathalfa}

\usepackage[margin=2.3cm]{geometry}
\usepackage{tikz-cd}
\usepackage{enumitem}
\setlist[enumerate,1]{font=\normalfont, label=(\roman*)}

\usepackage{hyperref}

\newtheorem{theorem}{Theorem}[section]

\newtheorem{lemma}[theorem]{Lemma}

\newtheorem{proposition}[theorem]{Proposition}

\newtheorem{setup}[theorem]{Setup}

\theoremstyle{definition}
\newtheorem{definition}[theorem]{Definition}

\newtheorem{remark}[theorem]{Remark}

\DeclareMathOperator{\cone}{cone}
\DeclareMathOperator{\ev}{ev}

\DeclareMathOperator{\Hom}{Hom}
\DeclareMathOperator{\Ext}{Ext}

\DeclareMathOperator{\cok}{cok}

\DeclareMathOperator{\Stab}{Stab}

\DeclareMathOperator{\Sing}{Sing}

\DeclareMathOperator{\Pic}{Pic}

\DeclareMathOperator{\RHom}{\mathbf{R}Hom}

\title[Stability conditions on a singular quadric threefold]{Stability conditions on a singular quadric threefold}  
\author[Tzu-Yang Chou]{Tzu-Yang Chou}

\begin{document}

\begin{abstract}
Let $X \subset \mathbb{P}^4$ be a quadric threefold with a single ordinary double point, and let $\mathcal{K}u(X)$ be its Kuznetsov component. In this paper, we construct a weak stability condition on Kuznetsov's categorical resolution $\widetilde{D} \subset \mathrm{D^b}(\widetilde{X})$, compatible with the Verdier localization $\mathbf{R}\pi_* \colon \widetilde{D} \to \mathrm{D^b}(X)$, and hence obtain a Bridgeland stability condition on $\mathrm{D^b}(X)$.
Restricting the construction, we obtain the corresponding statement for $\mathcal{K}u(X)$ and its categorical resolution $\widetilde{D}'$.
These can be viewed as a three-dimensional analogue of our previous result in~\cite{Cho25}.

We describe the geometry of the blow-up $\pi \colon \widetilde{X} \to X$ and obtain two semiorthogonal decompositions of $\mathrm{D^b}(\widetilde{X})$, arising from the projective bundle structure of $\widetilde{X}$ and from Kuznetsov's categorical resolution. Comparing them, we isolate an admissible subcategory $\widetilde{\mathcal{D}}\subset \mathrm{D^b}(\widetilde{X})$ resolving $\mathrm{D^b}(X)$ and show that it admits a full Ext-exceptional collection,
from which we construct the localization-compatible weak stability condition.

\end{abstract}

\maketitle

\section{Introduction}
Bridgeland stability conditions on triangulated categories, introduced in \cite{Bri07}, have become a central tool in the study of derived categories.
They are known to exist on curves \cite{Mac07} and on surfaces \cite{AB13}.
In dimension three, the tilt-stability approach of Bayer--Macr\`i--Toda \cite{BMT14} and its refinements provides a general framework for constructing Bridgeland stability conditions, and has led to existence results in a number of cases.

Very recently, Li proved in \cite{Li26} the existence of Bridgeland stability conditions on $\mathrm{D^b}(X)$ for any smooth projective variety $X$.
It is natural that we turn to study Bridgeland stability conditions on singular varieties.
Motivated by our previous work \cite{Cho25} on a singular surface and its resolution, it is natural to study stability conditions on a suitable categorical resolution first, and then compare them with stability conditions downstairs via the associated Verdier localization.

In this article, we extend our previous result in \cite{Cho25} to singular threefolds via Kuznetsov's categorical resolution \cite{Kuz08b},
and we first address the simplest such case, namely a quadric threefold with a single ordinary double point.

Our main result is the following:
\begin{theorem}\label{main1}
Let $X$ be a 1-nodal quadric threefold  and $\pi \colon \widetilde{X} \longrightarrow X$ be its blow-up. 

Then there exist a stability condition $\sigma_{\mathrm{D^b}(X)}=(Z_\mathcal{A}, \mathcal{A})$ on $\mathrm{D^b}(X)$,
and a weak stability condition $\sigma_{\widetilde{\mathcal{D}}} = (Z_{\widetilde{\mathcal{A}}},\widetilde{\mathcal{A}})$ on the categorical resolution $\widetilde{\mathcal{D}}$ of $\mathrm{D^b}(X)$,
such that they are related as follows.
\begin{enumerate}
    \item[(1)] $Z_{\widetilde{\mathcal{A}}} = Z_\mathcal{A} \circ \pi_\ast$;
    \item[(2)] $\mathbf{R}\pi_\ast \colon \widetilde{\mathcal{A}}\longrightarrow \mathcal{A}$ is an exact functor and preserves semistability.
\end{enumerate}
\end{theorem}

In order to construct stability conditions on the derived category of a singular threefold,
it is natural to first work on an appropriate ``resolution''.
For a nodal quadric threefold $X$, Kuznetsov's theory of categorical resolutions \cite{Kuz08b}
together with the construction of Kuznetsov--Shinder \cite{KS24} produce an admissible subcategory
$\widetilde{\mathcal{D}}\subset\mathrm{D^b}(\widetilde{X})$, 
equipped with a functor 
$\mathbf{R}\pi_\ast \colon \widetilde{\mathcal{D}} \longrightarrow \mathrm{D^b}(X)$ which is a Verdier localization,
with kernel generated by a single spherical object.

The category $\widetilde{\mathcal{D}}$ admits an explicit description thanks to the projective bundle structure of $\widetilde{X}$,
and it will serve as the starting point of our construction.

First, we exploit two natural semiorthogonal decompositions of $\mathrm{D^b}(\widetilde{X})$. 
The first arises from Orlov’s projective bundle formula for $\tilde{X}\cong \mathbb{P}_E(\mathcal{O}_E \oplus \mathcal{O}_E(-1,-1))$,
which yields a semiorthogonal decomposition of $\mathrm{D^b}(\widetilde{X})$ given by eight exceptional line bundles.
The second decomposition comes from Kuznetsov’s categorical resolution of the nodal quadric.

The key point is that, by performing a sequence of mutations relating these two decompositions,
we exhibit a full Ext-exceptional collection of length $6$ in the categorical resolution 
$\widetilde{\mathcal{D}}$ of $\mathrm{D^b}(X)$,
which greatly simplifies the construction of stability conditions.

More precisely, we have:
\begin{theorem}
The categorical resolution  $\widetilde{\mathcal{D}}$ of $\mathrm{D^b}(X)$ admits a full Ext-exceptional collection of length $6$.
\end{theorem}

Their extension closure immediately gives a heart $\mathcal{B}$ of bounded $t$-structure.
Nevertheless, it can be seen that no weak stability condition with heart $\mathcal{B}$ descends to $\mathrm{D^b}(X)$,
and we therefore find a suitable tilt to construct a new heart $\widetilde{\mathcal{A}}\subset \widetilde{\mathcal{D}}$,
so that its pushforward $\mathcal{A}:=\mathbf{R}\pi_\ast(\widetilde{\mathcal{A}})$ is a heart on $\mathrm{D^b}(X)$.
We can then define a localization-compatible central charge $Z_{\widetilde{\mathcal{A}}}$ 
and induce a stability condition  $\sigma_{\mathcal{A}}=(Z_{\mathcal{A}},\mathcal{A})$ on $\mathrm{D^b}(X)$.

The intermediate subcategory $\widetilde{\mathcal{D}}'$ appearing in the construction is also of independent interest: it is the categorical resolution of the Kuznetsov component $\mathcal{K}u(X)$.
Introduced by Kuznetsov in his study of semiorthogonal decompositions of Fano threefolds \cite{Kuz04,Kuz05,Kuz09}, such components have played a central role in the study of their derived categories and geometry.

Stability conditions on Kuznetsov components have also been extensively studied.
In the smooth case, they were constructed in \cite{BLMS} for Fano threefolds of Picard rank one and cubic fourfolds.
In our singular setting, the construction above can also be restricted to yield the corresponding statement for $\mathcal{K}u(X)$.

\subsection{Related work}
Our result fits into recent developments on Bridgeland stability conditions and moduli on Kuznetsov components of Fano varieties.

In the smooth case, stability conditions on Fano threefolds of Picard rank one and their Kuznetsov components are constructed in \cite{Li19} and \cite{BLMS}, respectively.

On the singular side, Kuznetsov and Shinder introduced the notion of categorical absorption of singularities and constructed, under natural hypotheses,
categorical absorptions for projective varieties with isolated ordinary double points \cite{KS24}.
They also studied the derived categories of 1-nodal prime Fano threefolds in \cite{KS25}.

Independently, in \cite{CGL+} the authors studied kernels of categorical resolutions of nodal singularities and proved that for nodal varieties the kernel is generated by a single spherical object, with applications to nodal cubic fourfolds.

Our weak stability condition on the categorical resolution also fits into the framework of partial compactifications of stability spaces by massless objects developed in \cite{BPPW22}. In their terminology it can be viewed as a lax stability condition lying on a boundary stratum,
with massless subcategory corresponding to the  kernel of the Verdier localization.

Finally, in our previous work \cite{Cho25} we constructed Bridgeland stability conditions on a singular surface and its resolution,
and the present paper can be viewed as a threefold counterpart.

\subsection{Organization of the paper}
In Section 2 we review background on stability conditions, hearts of $t$-structures, and tilting. 

Section 3 is devoted to the geometry of the nodal quadric threefold $X$ and its resolution $\widetilde{X}$; we describe the projective bundle structure of $\widetilde{X}$ and derive a semiorthogonal decomposition of $\mathrm{D^b}(\widetilde{X})$ by Orlov's formula. 

In Section 4, we introduce Kuznetsov’s categorical resolution $\widetilde{\mathcal{D}}$ of $\mathrm{D^b}(X)$ and isolate the admissible subcategory $\widetilde{\mathcal{D}}'$ which resolves the Kuznetsov component $\mathcal{K}u(X)$. We then perform a sequence of mutations, using the projective bundle decomposition from Section 3, to exhibit a full exceptional triple generating $\widetilde{\mathcal{D}}'$. 

In Section 5, we use this exceptional collection to construct a new heart $\widetilde{\mathcal{A}}$ on $\widetilde{\mathcal{D}}$ and prove that it descends to a heart $\mathcal{A}$ on $\mathrm{D^b}(X)$.
We also carry out the construction of the stability condition on $\mathrm{D^b}(X)$ and prove the main result.

Throughout the paper, we work over the complex numbers $\mathbb{C}$.

\section*{Acknowledgements}

The author would like to express his sincere gratitude to Arend Bayer for his invaluable help and many stimulating discussions throughout the development of this work.

Most of the mathematical results in this paper were obtained while the author was supported by the ERC Consolidator Grant Wall-CrossAG, no. 819864.
Some further results and the preparation of the manuscript were supported by the National Science and Technology Council of Taiwan under grant No. 114-2639-M-002-006-ASP.

\section{Preliminaries}
We first review some basic notions.
Throughout this section, let $\mathcal{C}$ be a triangulated category. 
We fix a finite rank lattice $\Lambda$ and a surjective group homomorphism $v\colon {\rm K}(\mathcal{C}) \longrightarrow \Lambda$. 

\begin{definition}
A weak stability function (resp. stability function) on  an abelian category $\mathcal{H}$ is a group homomorphism $Z\colon {\rm K}(\mathcal{H}) \longrightarrow \mathbb{C}$ such that for any nonzero object $E\in \mathcal{H}$, we have $\Im Z(E) \geq 0$ with $\Im Z(E) = 0 \Rightarrow \Re Z(E) \leq 0$ (resp. $ \Re Z(E) < 0$).
\end{definition}

\begin{definition}\label{stability}
A weak stability condition (resp. Bridgeland stability condition) on $\mathcal{C}$ is a pair $\sigma = (Z,\mathcal{H})$ consisting of a group homomorphism (called the central charge of $\sigma$) $Z \colon \Lambda \longrightarrow \mathbb{C}$ and a heart $\mathcal{H}$ of a bounded t-structure on $\mathcal{C}$, such that the following conditions hold:
\begin{enumerate}
    \item[(a)] The composition $Z \circ v \colon {\rm K}(\mathcal{H})={\rm K}(\mathcal{C}) \longrightarrow \Lambda \longrightarrow \mathbb{C}$ is a weak stability function (resp. stability function) on $\mathcal{H}$. This gives a notion of slope: for any $E \in \mathcal{H}$, we set
    \begin{equation*}
    \mu_\sigma(E)=\mu_Z(E):=
    \begin{cases}
    \frac{-\Re Z(E)}{\Im Z(E)}, & \text{if } \Im Z(E)>0, \\
    +\infty, & \text{if } \Im Z(E)=0.
    \end{cases}
    \end{equation*}
    We say that an object $0\neq E \in \mathcal{H}$ is $\sigma$-semistable (resp. $\sigma$-stable) if for every nonzero proper subobject $F \subset E$, we have $\mu_\sigma (F) \leq \mu_\sigma (E)$ (resp. $\mu_\sigma (F) < \mu_\sigma (E)$).
    
    \item[(b)] Any nonzero object $E \in \mathcal{H}$ admits a Harder--Narasimhan filtration in $\sigma$-semistable ones,
    that is, for any $0\neq E \in \mathcal{H}$, there is a filtration 
    $$0=E_0 \subset E_1\subset \dots \subset E_{n-1} \subset E_n=E,$$ such that the quotients $A_i:=E_i/E_{i-1}$ are $\sigma$-semistable (called the HN factors) and we have an inequality of their slopes $\mu_\sigma(A_1)> \dots >\mu_\sigma(A_n)$.
    \item[(c)] (Support property) There exists a quadratic form $Q$ on $\Lambda \otimes \mathbb{R}$ such that $Q|_{\ker Z}$ is negative definite, and $Q(v(E))\geq 0$, for all $\sigma$-semistable $E \in \mathcal{H}$.
\end{enumerate}
%Any pair $\sigma=(Z,\mathcal{H})$ satisfying the conditions (a) and (b) is said to be a weak pre-stability condition (resp. pre-stability condition).
\end{definition}

We now briefly review the theory of torsion pairs and tilting introduced in \cite{HRS96}.

\begin{definition} 
Given an abelian category $\mathcal{H}$, a torsion pair in $\mathcal{H}$ is a pair $(\mathcal{T} , \mathcal{F})$ of full additive subcategories satisfying the following conditions:
\begin{enumerate}
    \item[(1)] $\Hom(\mathcal{T} , \mathcal{F})  = 0$
    \item[(2)] For every $E \in \mathcal{H}$, there exists a short exact sequence
    
$$ 0 \longrightarrow T \longrightarrow  E \longrightarrow  F \longrightarrow 0$$
with $ T \in \mathcal{T}$ and $F \in \mathcal{F}$.
\end{enumerate}
\end{definition}

Note that a torsion pair is completely determined by either its torsion part or its torsion-free part. 
More precisely, if $(\mathcal{T} , \mathcal{F})$ is a torsion pair, then we have $\mathcal{F}=\mathcal{T}^\perp$ and $\mathcal{T}={^\perp\mathcal{F}}$.

If we have a torsion pair in an abelian category, we can obtain a tilted heart. 
More precisely, we have the following result.
\begin{theorem}[{\normalfont\cite{HRS96}}]\label{HRStilt}
Let $\mathcal{H}$ be the heart of a bounded t-structure in $\mathcal{C}$ with a torsion pair $(\mathcal{T}, \mathcal{F})$. 
The category
$$\mathcal{H}^\sharp := \{ E \in \mathcal{C} \mid H^i_{\mathcal{H}}(E) = 0 \ {\rm if} \ i \neq 0, -1, H^0_{\mathcal{H}}(E) \in \mathcal{T} , H^{-1}_{\mathcal{H}}(E) \in  \mathcal{F} \}$$ 
is the heart of a bounded t-structure on $\mathcal{C}$.
\end{theorem}

We also recall (weak) stability conditions arising from full exceptional collections.

\begin{definition}
An exceptional collection $\{E_1, \dots, E_n\}$ of a triangulated category $\mathcal{C}$ is called Ext-exceptional if $\Hom^{\leq 0 }(E_i,E_j)=0$ for all $i \neq j$.
\end{definition}

\begin{lemma}[{\normalfont\cite[Lemma 3.14]{Mac07}}]\label{Mac07}
Let $\{E_1,\ldots,E_n\}$ be a full Ext-exceptional collection on a triangulated category $\mathcal{C}$.
Then $[E_1,\ldots,E_n]$ is the heart of a bounded t-structure in $\mathcal{C}$,
where $[]$ denotes the extension closure.
\end{lemma}

By this lemma, for a full exceptional collection $\{E_1, \dots, E_n\}$ of $\mathcal{C}$, if we can find integers $p_1, \dots , p_n$ such that the exceptional collection $\{E_1[p_1], \dots , E_n[p_n] \}$ is Ext,
which is always possible when $\mathcal C=\mathrm{D^b}(\mathcal{H})$ for some abelian category $\mathcal{H}$,
then the extension closure $\mathcal{A}$ of $\{E_1[p_1], \dots , E_n[p_n] \}$ is a heart of bounded t-structure on $\mathcal{C}$.

Moreover, since such heart $\mathcal{A}$ is of finite length and the simple objects in $\mathcal{A}$ are exactly $E_i[p_i]$,
we can define a stability function (resp. weak stability function) $Z$ by designating the values of $E_i[p_i]$'s in $\mathbb{H}:= \{ z \in \mathbb{C} \mid z = |z| e^{i\pi\phi}, \phi \in (0,1]\}$
(resp. $\mathbb{H}\cup \{0\}$),
and in this case, $(Z,\mathcal{A})$ admits Harder--Narasimhan property automatically.

\section{Geometry of the nodal quadric threefold and its resolution}

In this section, we study the geometry of $X$ and $\widetilde{X}$.
We fix the following notation throughout the paper.

\begin{setup}\label{setup:nodal-quadric}
Let $X \subset \mathbb{P}^4$ be a quadric threefold with a single ordinary double point $x_0$.
Let $\pi \colon \widetilde{X} \longrightarrow X$ be the blow-up at $x_0$,
and let $\epsilon \colon E \simeq \mathbb{P}^1\times\mathbb{P}^1 \hookrightarrow \widetilde{X}$ be the exceptional divisor.

We denote by $H$ the primitive ample divisor such that $K_X = -3H$, 
and by $h,k$ the divisor classes of $\mathcal{O}_E(1,0),\mathcal{O}_E(0,1)$ respectively.
\end{setup}

Note that $X$ is a Fano threefold of Picard rank $1$ and index $3$, with $\mathrm{K_{num}}(X) \simeq \mathbb{Z}^4$ and a semiorthogonal decomposition
\begin{equation*}
\mathrm{D^b}(X) = \langle
\mathcal{K}u(X),
\mathcal{O}_X, 
\mathcal{O}_X(H),
\mathcal{O}_X(2H) \rangle.
\end{equation*}
This in particular implies that $\mathrm{K_{num}}(\mathcal{K}u(X)) \simeq \mathbb{Z}$.

We start with the following geometric construction.

\begin{proposition}
$\widetilde{X}\simeq \mathbb{P}_E(\mathcal{O}_E\oplus \mathcal{O}_E(-1,-1))$.
\end{proposition}
\begin{proof}
We first choose a hyperplane $\mathbb{P}^3 \subseteq \mathbb{P}^4$ which does not contain the singularity $x_0$ and consider the projection through $x_0$.
This is a birational map $X \dashrightarrow \mathbb{P}^3$ which is resolved by the blow-up at $x_0$ and induces a map $f \colon \widetilde{X} \longrightarrow \mathbb{P}^3$. We denote its image by $Z = f(\widetilde{X}) = X \cap \mathbb{P}^3 \subseteq \mathbb{P}^3$,
and $\widetilde{X}$ is a $\mathbb{P}^1$-bundle over $Z$.

On the other hand, the exceptional divisor $E$ is isomorphic to $Z$ via the composition ($E \longrightarrow \widetilde{X} \longrightarrow Z$) of $\epsilon$ and $f$. 

In summary, the morphism $\rho:=(f\epsilon)^{-1}f  \colon (\widetilde{X} \longrightarrow Z \simeq E)$ is a $\mathbb{P}^1$-bundle, and both the composition ($E \simeq Z \hookrightarrow X \setminus \Sing(X) \hookrightarrow \widetilde{X}$)
and the inclusion ($\epsilon \colon E \longrightarrow \widetilde{X}$) give sections of $\rho$; 
moreover, these two sections are disjoint.

This implies that $\widetilde{X}$ is the projectification of the direct sum of the two line bundles corresponding to these two sections respectively, that is, $\mathcal{O}_E$ and $\mathcal{N}_{E/\widetilde{X}} \simeq \mathcal{O}_E(-1,-1)$.
\end{proof}

Note that $ \pi^\ast \mathcal{O}_X(1)$ is isomorphic to the relative $\mathcal{O}(1)$ of $\rho$,
and hence by abusing the notation we denote by $H$ both the divisor on $X$ generating $\Pic(X)$ and that on $\widetilde{X}$ corresponding to the relative $\mathcal{O}(1)$ of $\rho$,
and by $h,k$ the class of $\mathcal{O}_E(1,0),\mathcal{O}_E(0,1)$ and their pullbacks to $\widetilde{X}$ respectively. 

We now recall Orlov's well-known projective bundle formula.
\begin{theorem}\cite{Orl92}\label{Orlov}
Let $M$ be a smooth projective variety, and let $V$ be a vector bundle of rank $r$
on $M$. 
Denote by $\rho \colon \mathbb{P}_M(V) \longrightarrow M$
the associated projective bundle, and by $\mathcal{O}_{\mathbb{P}_M(V)}(1)$ its
relative $\mathcal{O}(1)$.
Then there is a complete semiorthogonal decomposition of $D^b(\mathbb{P}_M(V))$:
\[
\mathrm{D^b}(\mathbb{P}_M(V))=
\langle \rho^\ast \mathrm{D^b}(M), \rho^\ast \mathrm{D^b}(M)\otimes\mathcal{O}_{\mathbb{P}_M(V)}(1),\dots,\rho^\ast \mathrm{D^b}(M)\otimes\mathcal{O}_{\mathbb{P}_M(V)}(r-1)\rangle.
\]
\end{theorem}

Applying Theorem \ref{Orlov} to the projective bundle $\rho \colon \widetilde{X}=\mathbb{P}_E(\mathcal{O}_E\oplus \mathcal{O}_E(-1,-1)) \longrightarrow E$,
we obtain the following result.
\begin{lemma}\label{O1}
There is a semiorthogonal decomposition
\begin{equation}\label{SOD1}
\mathrm{D^b}(\widetilde{X}) = \langle \mathcal{O}_{\widetilde{X}}, \mathcal{O}_{\widetilde{X}}(h),
\mathcal{O}_{\widetilde{X}}(k),
\mathcal{O}_{\widetilde{X}}(h+k),
\mathcal{O}_{\widetilde{X}}(H),
\mathcal{O}_{\widetilde{X}}(H+h),
\mathcal{O}_{\widetilde{X}}(H+k),
\mathcal{O}_{\widetilde{X}}(H+h+k) \rangle
\end{equation}

In particular, we have $\mathrm{K_{num}}(\widetilde{X})\simeq \mathbb{Z}^8$.
\end{lemma}

We will also need to compute the Serre functor on the triangulated category $\mathrm{D^b}(\widetilde{X})$ in order to compare semiorthogonal decompositions later.

\begin{lemma}\label{Serre}
\begin{enumerate}
    \item The Serre functor $S_{\widetilde{X}}$ on $\mathrm{D^b}(\widetilde{X})$ is given by $S_{\widetilde{X}}(\cdot) = \cdot \otimes \mathcal{O}_{\widetilde{X}}(-2H-h-k)[3]$. 
    \item $E= H-h-k$.
\end{enumerate}
\end{lemma}
\begin{proof}
For (i), as $\widetilde{X}$ is smooth, we only need to compute its canonical divisor $K_{\widetilde{X}}$.

Since $\widetilde{X} \longrightarrow E$ is a projective bundle, we can compute the relative tangent sheaf,
and hence one can also use the following short exact sequence to obtain the canonical divisor $K_{\widetilde{X}}$. 
$$ 0 \longrightarrow \mathcal{T}_\rho \longrightarrow \mathcal{T}_{\widetilde{X}} \longrightarrow \rho^\ast \mathcal{T}_{E} \longrightarrow 0$$

Taking the determinant line bundles of the short exact sequence gives us that 
\begin{equation*} 
\begin{split}
\omega_{\widetilde{X}} = (\bigwedge^3 \mathcal{T}_{\widetilde{X}})^\vee &\simeq (\bigwedge^1 \mathcal{T}_\rho)^\vee \otimes (\bigwedge^2 \rho^\ast \mathcal{T}_{E})^\vee\\ 
&= [\bigwedge^2 (\mathcal{O}_{\widetilde{X}} \otimes \rho^\ast (\mathcal{O}_E \oplus \mathcal{O}_E(-1,-1)))]^\vee \otimes \rho^\ast \omega_E\\
&= [\mathcal{O}_{\widetilde{X}}(2H) \otimes (\bigwedge^2 \rho^\ast(\mathcal{O}_E \oplus \mathcal{O}_E(-1,-1)))]^\vee \otimes \mathcal{O}_{\widetilde{X}}(-2h-2k)\\
&= \mathcal{O}_{\widetilde{X}}(-2H-h-k).
\end{split}
\end{equation*}

For (ii), we write $K_{\widetilde{X}} = \pi^\ast K_X +aE$,
then by adjunction formula, we obtain that
\[(\pi^\ast K_X + aE +E)|_E = K_E = -2h-2k.\]

Since $\pi^\ast K_X = -3H$ is trivial on $E$ and $E|_E = -h-k$, we conclude that $a=1$, 
and therefore $E =K_{\widetilde{X}} - \pi^\ast K_X = -2H-h-k -(-3H) = H-h-k$.
\end{proof}

\section{Categorical resolution and its full exceptional collection}

We now turn to Kuznetsov's categorical resolution of $X$,
which is a component of $\mathrm{D^b}(\widetilde{X})$. 
We will compare the semiorthogonal decomposition produced by categorical resolution with the one given in Lemma \ref{O1},
and work out a full exceptional collection of the categorical resolution of $\mathcal{K}u(X)$.

In \cite{Kuz08b}, Kuznetsov proved the following result for the resolution $\pi \colon \widetilde{X} \longrightarrow X$.

\begin{theorem}\cite[Theorem 4.4]{Kuz08b}
Let $\pi : \widetilde{X} \to X$ be a resolution of rational singularities.
Assume that $E \subset \widetilde{X}$ is the exceptional locus of $\pi$,
$Z = \pi(E)$, and $p = \pi|_E : E \to Z$. 
Let $\epsilon : E \hookrightarrow \widetilde{X}$ denote the inclusion.

Suppose that we have a dual Lefschetz decomposition
\[
  \mathrm{D^b}(E)=
  \langle
    \mathcal{A}_{m-1}(1-m),\,
    \mathcal{A}_{m-2}(2-m),\,
    \dots,\,
    \mathcal{A}_{1}(-1),\,
    \mathcal{A}_{0}
  \rangle
\]
with respect to the line bundle
$\mathcal{O}_{E}(1) = \mathcal{N}^{*}_{E/\widetilde{X}}$,
and that the subcategory
$\mathcal{A}_{0} \subset \mathrm{D^b}(E)$
contains $p^{*}(\mathrm{D^{perf}}(Z))$.
Define
$\widetilde{\mathcal{D}}:=
\{ F \in \mathrm{D^b}(\widetilde{X}) | \epsilon^\ast F \in \mathcal{A}_{0} \}
\subseteq \mathrm{D^b}(\widetilde{X}).$

Then we have a semiorthogonal decomposition
\[
  \mathrm{D^b}(\widetilde{X})
  =
  \langle
    \epsilon_\ast(\mathcal{A}_{m-1}(1-m)),
    \epsilon_\ast(\mathcal{A}_{m-2}(2-m)),
    \dots,
    \epsilon_\ast(\mathcal{A}_{1}(-1)),
    \widetilde{\mathcal{D}}
  \rangle
\]
and $\widetilde{\mathcal{D}}$ is a categorical resolution of $\mathrm{D^b}(X)$.
\end{theorem}

In our case, the exceptional divisor $E$ is the quadric surface $\mathbb{P}^1 \times \mathbb{P}^1$,
and the conormal bundle $\mathcal{O}_E(-E)$ is just $\mathcal{O}_E(h+k)$.
This immediately induces another semiorthogonal decomposition, given by
\begin{equation*}
\mathrm{D^b}(\widetilde{X}) = \langle \epsilon_\ast \mathcal{O}_E \otimes \mathcal{O}_{\widetilde{X}}(-2h-k), 
\epsilon_\ast \mathcal{O}_E \otimes \mathcal{O}_{\widetilde{X}}(-h-k), \widetilde{\mathcal{D}} \rangle
\end{equation*}
where $\widetilde{\mathcal{D}}$ is the categorical resolution of $X$.
We may further note that  $\pi^\ast({\rm D^{perf}}(X)) \subseteq \widetilde{\mathcal{D}}$.
Pulling back the exceptional collection $\{\mathcal{O}_X, \mathcal{O}_X(H),\mathcal{O}_X(2H)\}$ to $\widetilde{X}$,
we obtain the following semiorthogonal decomposition:
\begin{lemma}\label{K1}
There is a semiorthogonal decomposition
\begin{equation}\label{SOD2}
\mathrm{D^b}(\widetilde{X}) = \langle \mathcal{O}_E(-2h-k), 
\mathcal{O}_E(-h-k), \widetilde{\mathcal{D}}', \mathcal{O}_{\widetilde{X}},
\mathcal{O}_{\widetilde{X}}(H),
\mathcal{O}_{\widetilde{X}}(2H)
\rangle
\end{equation}
where $\widetilde{\mathcal{D}}' := \{A \in \widetilde{\mathcal{D}} \ | \Hom_{\widetilde{\mathcal{D}}}(\mathcal{O}_{\widetilde{X}}(iH),A) = 0 {\rm \ for \ } i=0,1,2\}$ is the semiorthogonal complement.
\end{lemma}

Looking at the decompositions (\ref{SOD1}) and (\ref{SOD2}), it is natural to hope that the categorical resolution $\widetilde{\mathcal{D}}'$ of $\mathcal{K}u(X)$ admits a $3$-term full exceptional collection.
We are going to prove this by a sequence of mutations.
We can compute the left mutations by computing the global sections of any sheaf on $\widetilde{X}$ by pushing forward to $E\simeq \mathbb{P}^1\times \mathbb{P}^1$ along $\rho$.

\begin{remark}
We know that $\mathrm{D^b}(E)$ admits a rectangular Lefschetz decomposition of length $2$,
and hence by \cite[Corollary 4.8]{Kuz08b}, we see that $\widetilde{\mathcal{D}}$ is a strongly crepant categorical resolution of $X$,
i.e. the relative Serre functor $S_{\widetilde{\mathcal{D}}/ \widetilde{X}}$ is the identity.

In particular, this implies that both the left and right adjoints of $\mathbf{R}\pi_\ast \colon \widetilde{\mathcal{D}} \longrightarrow \mathrm{D^b}(X)$ are given by the pullback $\mathbf{L}\pi^\ast \colon {\rm D^{perf}}(X) \longrightarrow \widetilde{\mathcal{D}}$, i.e. $\pi^!=\mathbf{L}\pi^\ast$.

The same argument shows that the categorical resolutions for other Picard rank one Fano threefolds with a single node are also strongly crepant,
even though in those cases we cannot express $E$ in terms of $H, h$ and $k$.
\end{remark}

Now we can start the computation of mutations.

\begin{proposition}\label{KvsO}
The triangulated subcategory $\widetilde{\mathcal{D}}'$ admits a full exceptional collection of length $3$.
More precisely, we have $\widetilde{\mathcal{D}}'=\langle \mathcal{O}_{\widetilde{X}}(-h),\mathbf{L}_{\mathcal{O}_E(-h)}\mathcal{O}_{\widetilde{X}}(-k),\mathbf{L}_{\mathcal{O}_E(-h)}\mathbf{L}_{\mathcal{O}_{\widetilde{X}}} \mathcal{O}_{\widetilde{X}}(H-k) \rangle$.
\end{proposition}
\begin{proof}
We divide the proof into steps.

\textbf{Step 1.}
We start with the semiorthogonal decomposition (\ref{SOD2}).
By Lemma \ref{K1} and applying the Serre functor, we get
\begin{equation}\label{SOD3}
\mathrm{D^b}(\widetilde{X}) = 
\langle  \widetilde{\mathcal{D}}', \mathcal{O}_{\widetilde{X}},
\mathcal{O}_{\widetilde{X}}(H),
\mathcal{O}_{\widetilde{X}}(2H),
\mathcal{O}_E(-h),\mathcal{O}_E
\rangle
\end{equation}.

Next, for $i=0,1,2$, we have
$$\RHom_{\widetilde{X}}( \mathcal{O}_{\widetilde{X}}(iH), \mathcal{O}_E(-h)) = \mathbf{R}\Gamma(\widetilde{X},\mathcal{O}_E(-iH-h)) = \mathbf{R}\Gamma (E, \mathcal{O}_E(-1,0)) = 0,$$
and hence we see that $\mathcal{O}_E(-h)$ is completely orthogonal to $\mathcal{O}_{\widetilde{X}}(iH)$.

Mutating $\mathcal{O}_E(-h)$ to the left of these three objects yields
\begin{equation}\label{SOD4}
\mathrm{D^b}(\widetilde{X}) = 
\langle  \widetilde{\mathcal{D}}', \mathcal{O}_E(-h), \mathcal{O}_{\widetilde{X}},
\mathcal{O}_{\widetilde{X}}(H),
\mathcal{O}_{\widetilde{X}}(2H),
\mathcal{O}_E
\rangle
\end{equation}

\textbf{Step 2.} We now consider the left mutation $\mathbf{L}_{\mathcal{O}_{\widetilde{X}}(2H)}\mathcal{O}_E$.

Note that $\RHom_{\widetilde{X}}(\mathcal{O}_{\widetilde{X}}(2H), \mathcal{O}_E) \simeq \mathbf{R}\Gamma(\widetilde{X}, \mathcal{O}_E) \simeq \mathbb{C}$.
Thus, the distinguished triangle defining the left mutation $\mathbf{L}_{\mathcal{O}_{\widetilde{X}}(2H)}\mathcal{O}_E$ is exactly induced by the short exact sequence
$$ 0 \longrightarrow  \mathcal{O}_{\widetilde{X}}(2H-E) \longrightarrow \mathcal{O}_{\widetilde{X}}(2H) \longrightarrow \mathcal{O}_E(2H) \longrightarrow 0$$
and hence $\mathbf{L}_{\mathcal{O}_{\widetilde{X}}(2H)}\mathcal{O}_E = \mathbf{L}_{\mathcal{O}_{\widetilde{X}}(2H)}\mathcal{O}_E(2H)= \mathcal{O}_{\widetilde{X}}(2H-E)[1] =\mathcal{O}_{\widetilde{X}}(H+h+k)[1]$,
where the final equality follows from Lemma \ref{Serre}(ii).

In order to move $\mathcal{O}_{\widetilde{X}}(H+h+k)$ to the far right,
we compute the left mutation 
$\mathbf{L}_{\mathcal{O}_{\widetilde{X}}(H+h+k)}\mathcal{O}_{\widetilde{X}}(2H)$.
For this, we need to compute $$\RHom_{\widetilde{X}} (\mathcal{O}_{\widetilde{X}}(H+h+k), \mathcal{O}_{\widetilde{X}}(2H)) = \mathbf{R}\Gamma(\widetilde{X},\mathcal{O}_{\widetilde{X}}(H-h-k)) = \mathbf{R}\Gamma(\widetilde{X},\mathcal{O}_{\widetilde{X}}(E)) \simeq \mathbb{C},$$

Similarly, the left mutation is defined by the sequence
$$ 0 \longrightarrow  \mathcal{O}_{\widetilde{X}}(H+h+k) \longrightarrow \mathcal{O}_{\widetilde{X}}(H+h+k+E) \longrightarrow \mathcal{O}_E(H+h+k+E) \longrightarrow 0,$$
so $\mathbf{L}_{\mathcal{O}_{\widetilde{X}}(H+h+k)}\mathcal{O}_{\widetilde{X}}(2H) = \mathbf{L}_{\mathcal{O}_{\widetilde{X}}(H+h+k)}\mathcal{O}_{\widetilde{X}}(H+h+k+E) =\mathcal{O}_E(H+h+k+E) = \mathcal{O}_E(2H) = \mathcal{O}_E$.

Now our semiorthogonal decomposition becomes
\begin{align}
\begin{split}
\mathrm{D^b}(\widetilde{X}) 
&= \langle  \widetilde{\mathcal{D}}', \mathcal{O}_E(-h), \mathcal{O}_{\widetilde{X}},
\mathcal{O}_{\widetilde{X}}(H),
\mathcal{O}_{\widetilde{X}}(H+h+k),
\mathcal{O}_{\widetilde{X}}(2H)
\rangle
\\
&=\langle  \widetilde{\mathcal{D}}', \mathcal{O}_E(-h), \mathcal{O}_{\widetilde{X}},
\mathcal{O}_{\widetilde{X}}(H),
\mathcal{O}_E,
\mathcal{O}_{\widetilde{X}}(H+h+k)
\rangle    
\end{split}
\end{align}

Similarly, we have $\mathbf{L}_{\mathcal{O}_{\widetilde{X}}(H)}\mathcal{O}_E = \mathbf{L}_{\mathcal{O}_{\widetilde{X}}(H)}\mathcal{O}_E(H) = \mathcal{O}_{\widetilde{X}}(H-E)[1] =\mathcal{O}_{\widetilde{X}}(h+k)[1]$.
This implies that
\begin{equation}\label{SOD6}
\mathrm{D^b}(\widetilde{X}) = \langle  \widetilde{\mathcal{D}}', \mathcal{O}_E(-h), \mathcal{O}_{\widetilde{X}},
\mathcal{O}_{\widetilde{X}}(h+k),
\mathcal{O}_{\widetilde{X}}(H),
\mathcal{O}_{\widetilde{X}}(H+h+k)
\rangle
\end{equation}

\textbf{Step 3.} On the other hand, starting from the semiorthogonal decomposition (\ref{SOD1}),
we first compute the left mutation $\mathbf{L}_{\mathcal{O}_{\widetilde{X}}}\mathcal{O}_{\widetilde{X}}(h)$.
Note that $\RHom_{\widetilde{X}}(\mathcal{O}_{\widetilde{X}},\mathcal{O}_{\widetilde{X}}(h))=\mathbf{R}\Gamma(\mathbb{P}^1,\mathcal{O}_{\mathbb{P}^1}(1))=\mathbb{C}^2$

Pulling back the sequence $0 \longrightarrow \mathcal{O}(-1) \longrightarrow \mathcal{O}^{\oplus 2} \longrightarrow \mathcal{O}(1) \longrightarrow 0  $ on $\mathbb{P}^1$  to $\widetilde{X}$,
we obtain
$$ 0 \longrightarrow \mathcal{O}_{\widetilde{X}}(-h) \longrightarrow \mathcal{O}_{\widetilde{X}}\otimes \mathbb{C}^2 \longrightarrow \mathcal{O}_{\widetilde{X}}(h) \longrightarrow 0,$$
and hence the cone of the evaluation map $\ev \colon \mathcal{O}_{\widetilde{X}}\otimes \RHom_{\widetilde{X}}(\mathcal{O}_{\widetilde{X}},\mathcal{O}_{\widetilde{X}}(h)) \longrightarrow \mathcal{O}_{\widetilde{X}}(h)$ is isomorphic to $\mathcal{O}_{\widetilde{X}}(-h)[1]$.

It follows that $\mathbf{L}_{\mathcal{O}_{\widetilde{X}}}\mathcal{O}_{\widetilde{X}}(h) = \mathcal{O}_{\widetilde{X}}(-h)[1]$ and similarly $\mathbf{L}_{\mathcal{O}_{\widetilde{X}}}\mathcal{O}_{\widetilde{X}}(k) = \mathcal{O}_{\widetilde{X}}(-k)[1]$.
Tensoring the line bundle $\mathcal{O}_{\widetilde{X}}(H)$,
we also obtain \[\mathbf{L}_{\mathcal{O}_{\widetilde{X}}(H)}\mathcal{O}_{\widetilde{X}}(H+h) = \mathcal{O}_{\widetilde{X}}(H-h)[1] \quad \text{and}
\quad \mathbf{L}_{\mathcal{O}_{\widetilde{X}}(H)}\mathcal{O}_{\widetilde{X}}(H+k) = \mathcal{O}_{\widetilde{X}}(H-k)[1].\] 

Moreover, we observe that $\mathcal{O}_{\widetilde{X}}(H-h),\mathcal{O}_{\widetilde{X}}(H-k)$ and $\mathcal{O}_{\widetilde{X}}(h+k)$ are mutually completely orthogonal.
Indeed, we have 
\begin{align*}
\begin{split}
\RHom_{\widetilde{X}} (\mathcal{O}_{\widetilde{X}}(h+k), \mathcal{O}_{\widetilde{X}}(H-h))&= \mathbf{R}\Gamma(\widetilde{X}, \mathcal{O}_{\widetilde{X}}(H-2h-k))\\
&= \mathbf{R}\Gamma(E, \mathcal{O}_{E}(-2h-k) \oplus \mathcal{O}_E(-h))\\
&=0, 
\end{split}
\end{align*}
and similarly $\RHom_{\widetilde{X}} (\mathcal{O}_{\widetilde{X}}(h+k), \mathcal{O}_{\widetilde{X}}(H-k))=0$.

Therefore, the decomposition (\ref{SOD1}) becomes
\begin{align}\label{SOD7}
\begin{split}
\mathrm{D^b}(\widetilde{X}) &= \langle \mathcal{O}_{\widetilde{X}}(-h), \mathcal{O}_{\widetilde{X}}(-k),
\mathcal{O}_{\widetilde{X}},
\mathcal{O}_{\widetilde{X}}(h+k),
\mathcal{O}_{\widetilde{X}}(H-h),
\mathcal{O}_{\widetilde{X}}(H-k),
\mathcal{O}_{\widetilde{X}}(H),
\mathcal{O}_{\widetilde{X}}(H+h+k) \rangle \\
&= \langle \mathcal{O}_{\widetilde{X}}(-h), \mathcal{O}_{\widetilde{X}}(-k),
\mathcal{O}_{\widetilde{X}},
\mathcal{O}_{\widetilde{X}}(H-h),
\mathcal{O}_{\widetilde{X}}(H-k),
\mathcal{O}_{\widetilde{X}}(h+k),
\mathcal{O}_{\widetilde{X}}(H),
\mathcal{O}_{\widetilde{X}}(H+h+k) \rangle
\end{split}
\end{align}

We focus on the subcategory $\langle \mathcal{O}_{\widetilde{X}},
\mathcal{O}_{\widetilde{X}}(h+k),
\mathcal{O}_{\widetilde{X}}(H), \mathcal{O}_{\widetilde{X}}(H+h+k) \rangle^\perp$.
Comparing (\ref{SOD6}) and (\ref{SOD7}), we deduce that $\langle \widetilde{\mathcal{D}}' , \mathcal{O}_E(-h) \rangle = \langle \mathcal{O}_{\widetilde{X}}(-h), \mathcal{O}_{\widetilde{X}}(-k), 
\mathbf{L}_{\mathcal{O}_{\widetilde{X}}}\mathcal{O}_{\widetilde{X}}(H-h),
\mathbf{L}_{\mathcal{O}_{\widetilde{X}}}\mathcal{O}_{\widetilde{X}}(H-k) \rangle$.

\textbf{Step 4.} We claim that $\mathbf{L}_{\mathcal{O}_{\widetilde{X}}(-k)}\mathbf{L}_{\mathcal{O}_{\widetilde{X}}}\mathcal{O}_{\widetilde{X}}(H-h) = \mathcal{O}_E(-h)$.

To prove this, we note that $\RHom_{\widetilde{X}}(\mathcal{O}_{\widetilde{X}},\mathcal{O}_{\widetilde{X}}(H-h)) = \mathbb{C}^2$.
Consider the following two short exact sequences:
$$ 0 \longrightarrow \mathcal{O}_{\widetilde{X}}(-k) \longrightarrow \mathcal{O}_{\widetilde{X}}\otimes \mathbb{C}^2 \longrightarrow \mathcal{O}_{\widetilde{X}}(k) \longrightarrow 0$$
$$ 0 \longrightarrow \mathcal{O}_{\widetilde{X}}(H-h-E) \longrightarrow \mathcal{O}_{\widetilde{X}}(H-h)  \longrightarrow \mathcal{O}_E(H-h) \longrightarrow 0$$

The evaluation map $\ev \colon \mathcal{O}_{\widetilde{X}}\otimes \mathbb{C}^2 \longrightarrow \mathcal{O}_{\widetilde{X}}(H-h)$ is exactly the composition $$\mathcal{O}_{\widetilde{X}}\otimes \mathbb{C}^2 \longrightarrow \mathcal{O}_{\widetilde{X}}(k) =\mathcal{O}_{\widetilde{X}}(H-h-E) \longrightarrow \mathcal{O}_{\widetilde{X}}(H-h),$$
and hence $\cone(\ev)=\mathbf{L}_{\mathcal{O}_{\widetilde{X}}}\mathcal{O}_{\widetilde{X}}(H-h)$ is given by a two-term complex $G$ with $H^0(G) = \mathcal{O}_E(-h)$ and $H^{-1}(G) = \mathcal{O}_{\widetilde{X}}(-k)$.
In particular, it fits into the distinguished triangle
\begin{equation}\label{DT0}\tag{$\dagger$}
\mathcal{O}_{\widetilde{X}}(-k)[1] \longrightarrow \mathbf{L}_{\mathcal{O}_{\widetilde{X}}}\mathcal{O}_{\widetilde{X}}(H-h) \longrightarrow \mathcal{O}_E(-h).
\end{equation}

Applying the triangulated functor $\mathbf{L}_{\mathcal{O}_{\widetilde{X}}(-k)}$ to $\eqref{DT0}$,
it shows that \[\mathbf{L}_{\mathcal{O}_{\widetilde{X}}(-k)}\mathbf{L}_{\mathcal{O}_{\widetilde{X}}}\mathcal{O}_{\widetilde{X}}(H-h) = \mathbf{L}_{\mathcal{O}_{\widetilde{X}}(-k)}\mathcal{O}_E(-h).\] 

Moreover,
\[
\RHom_{\widetilde{X}}(\mathcal{O}_{\widetilde{X}}(-k),
\mathcal{O}_E(-h))
= \mathbf{R}\Gamma(E,\mathcal{O}_E(-h+k))
= \mathbf{R}\Gamma(E,\mathcal{O}_E(-1,1))=0,
\]
so $\mathbf{L}_{\mathcal{O}_{\widetilde{X}}(-k)}\mathcal{O}_E(-h)
=\mathcal{O}_E(-h)$,
which proves our claim.

Finally, mutating $\mathcal{O}_{\widetilde{X}}(-k)$ and $\mathbf{L}_{\mathcal{O}_{\widetilde{X}}}\mathcal{O}_{\widetilde{X}}(H-k)$ through $\mathcal{O}_E(-h)$ proves the assertion.
\end{proof}

We now use this full exceptional collection to describe the kernel of the induced map between numerical Grothendieck groups.

From now on, we fix the following notations: $\mathscr{F}:=\mathbf{L}_{\mathcal{O}_E(-h)}\mathbf{L}_{\mathcal{O}_{\widetilde{X}}} \mathcal{O}_{\widetilde{X}}(H-k)$,
$\mathscr{G}:=\mathbf{L}_{\mathcal{O}_E(-h)}\mathcal{O}_{\widetilde{X}}(-k)$,
and $\mathscr{E} :=\mathbf{L}_{\mathcal{O}_E(-h)}\mathcal{O}_E(-k)$.

\begin{lemma}\label{kernel}
The kernel of the pushforward $\pi_\ast \colon \mathrm{K_{num}}(\widetilde{\mathcal{D}}) \longrightarrow \mathrm{K_{num}}(X)$ is generated by $[\mathscr{E}]$ and $[\mathscr{G}]+[\mathscr{F}]$.
\end{lemma}
\begin{proof}
Note first that $\ker(\pi_\ast|_{\widetilde{{\mathcal{D}}}})$ coincides with $\ker (\pi_\ast \colon \mathrm{K_{num}}(\widetilde{\mathcal{D}}') \longrightarrow \mathrm{K_{num}}(\mathcal{K}u(X))$
and recall that $\mathrm{K_{num}}(\widetilde{\mathcal{D}}')=\mathbb{Z}^3$ and $\mathrm{K_{num}}(\mathcal{K}u(X))=\mathbb{Z}$,
which implies that the kernel $\ker(\pi_\ast|_{\widetilde{{\mathcal{D}}}})$ has rank $2$.

For $\mathscr{G}$, Serre duality gives 
\begin{equation*}
\begin{split}
\RHom_{\widetilde{X}}(\mathcal{O}_E(-h),\mathcal{O}_{\widetilde{X}}(-k)) &\simeq \RHom_{\widetilde{X}}(\mathcal{O}_{\widetilde{X}}(-k),\mathcal{O}_E(-2h-k-2H))^\vee[-3] \\
&\simeq \mathbf{R}\Gamma(E,\mathcal{O}_E(-2,0))^\vee[-3]\\
&= \mathbb{C}[-2],     
\end{split}
\end{equation*}
and hence $\mathscr{G}$ fits into the distinguished triangle
\begin{equation}\label{DT1}\tag{$\Delta$}
\mathcal{O}_E(-h)[-2] \longrightarrow \mathcal{O}_{\widetilde{X}}(-k) \longrightarrow \mathscr{G}.
\end{equation}

In particular, $\mathscr{G}$ is a two-term complex with 
$H^1(\mathscr{G}) = \mathcal{O}_E(-h)$ and 
$H^0(\mathscr{G}) = \mathcal{O}_{\widetilde{X}}(-k)$.

For $\mathscr{F}$, interchanging $h$ and $k$ in \eqref{DT0},
we obtain a distinguished triangle
$$\mathcal{O}_{\widetilde{X}}(-h)[1] \longrightarrow \mathbf{L}_{\mathcal{O}_{\widetilde{X}}} \mathcal{O}_{\widetilde{X}}(H-k) \longrightarrow \mathcal{O}_E(-k). $$
Applying the triangulated functor $\mathbf{L}_{\mathcal{O}_E(-h)}$ we see that $\mathscr{F}$ fits into
\begin{equation}\label{DT2}\tag{$\Delta'$}
\mathcal{O}_{\widetilde{X}}(-h)[1] \longrightarrow \mathscr{F}\longrightarrow \mathbf{L}_{\mathcal{O}_E(-h)}\mathcal{O}_E(-k).
\end{equation}

Since $\mathbf{R}\pi_\ast(\mathcal{O}_E(-k)) = \mathbf{R}\pi_\ast(\mathcal{O}_E(-h)) = 0$, 
applying $\mathbf{R}\pi_\ast$ to $\eqref{DT1}$ and $\eqref{DT2}$ gives
\[
\mathbf{R}\pi_\ast \mathscr{G} \simeq \mathbf{R}\pi_\ast \mathcal{O}_{\widetilde{X}}(-k),\qquad
\mathbf{R}\pi_\ast \mathscr{F} \simeq \mathbf{R}\pi_\ast \mathcal{O}_{\widetilde{X}}(-h)[1].
\]
Moreover, $[\mathbf{R}\pi_\ast \mathcal{O}_{\widetilde{X}}(-h)] = [\mathbf{R}\pi_\ast \mathcal{O}_{\widetilde{X}}(-k)]$ in 
$\mathrm{K_{num}}(\mathcal{K}u(X))$,
and hence $\pi_\ast([\mathscr{F}] + [\mathscr{G}]) = 0$.

On the other hand, consider the short exact sequence:
\[
0 \longrightarrow \mathcal{O}_{\widetilde X}(-E-h)
  \longrightarrow \mathcal{O}_{\widetilde X}(-h)
  \longrightarrow \mathcal{O}_E(-h)
  \longrightarrow 0.
\]

Applying the functor $\RHom_{\widetilde X}(-,\epsilon_\ast\mathcal{O}_E(-k))$, 
we obtain 
\[
\RHom_{\widetilde X}(\mathcal{O}_E(-h),\mathcal{O}_E(-k))
\longrightarrow
\RHom_{\widetilde X}(\mathcal{O}_{\widetilde X}(-h),\mathcal{O}_E(-k))
\longrightarrow
\RHom_{\widetilde X}(\mathcal{O}_{\widetilde X}(-E-h),\mathcal{O}_E(-k)).
\]

We compute
\[
\RHom_{\widetilde X}(\mathcal{O}_{\widetilde X}(-h),\mathcal{O}_E(-k))
\simeq
\mathbf{R}\Gamma(E,\mathcal{O}_E(h-k))=0,
\]
\[
\RHom_{\widetilde X}(\mathcal{O}_{\widetilde X}(-E-h),\mathcal{O}_E(-k))
\simeq
\mathbf{R}\Gamma(E,\mathcal{O}_E(-2k))\simeq \mathbb{C}[-1].
\]
Hence $\RHom_{\widetilde X}(\mathcal{O}_E(-h),\mathcal{O}_E(-k))
\simeq \mathbb{C}[-2].$

The left mutation of $\mathcal{O}_E(-k)$ through $\mathcal{O}_E(-h)$ fits into a distinguished triangle
\begin{equation}\label{DT3}\tag{$\Delta''$}
\mathcal{O}_E(-h)[-2]
\longrightarrow
\mathcal{O}_E(-k)
\longrightarrow
\mathscr{E}.
\end{equation}
In particular, we have $\pi_\ast([\mathscr{E}]) = \pi_\ast([\mathcal{O}_E(-k)] - [\mathcal{O}_E(-h)]) = 0$.

Linear independence follows from Proposition \ref{KvsO},
which shows that $\{[\mathcal{O}_{\widetilde{X}}(-h)],[\mathscr{G}],[\mathscr{F}]\}$ is a basis of $\mathrm{K_{num}}(\widetilde{\mathcal{D}}')$,
together with the relation $[\mathscr{E}]=[\mathscr{F}]+[\mathcal{O}_{\widetilde{X}}(-h)]$.
\end{proof}

\section{Descent of the heart and construction of stability conditions}
Before constructing a weak stability condition on $\widetilde{\mathcal{D}}$,
we first need a localization-compatible heart of a bounded t-structure.
One approach is to use a full exceptional collection of the categorical resolution.

\begin{lemma}\label{Ext ec}
\begin{enumerate}
    \item The full exceptional collection $\{\mathcal{O}_{\widetilde{X}}(-h), \mathscr{G},\mathscr{F}[-2] \}$ of $\widetilde{\mathcal{D}}'$ is an Ext-exceptional collection.
    \item Moreover, the exceptional collection in (i) can be completed to a full Ext-exceptional collection $\{\mathcal{O}_{\widetilde{X}}(-h), \mathscr{G},\mathscr{F}[-2], \mathcal{O}_{\widetilde{X}}[p_0], \mathcal{O}_{\widetilde{X}}(H)[p_1], \mathcal{O}_{\widetilde{X}}(2H)[p_2]\}$ of $\widetilde{\mathcal{D}}$ for some integers $p_i$.
\end{enumerate}
\end{lemma}
\begin{proof}
Once (i) is proved, (ii) is automatic.

We need to show that $\Hom^i_{\widetilde{X}}(\mathcal{O}_{\widetilde{X}}(-h),\mathscr{G}),\Hom^{i-2}_{\widetilde{X}}(\mathcal{O}_{\widetilde{X}}(-h),\mathscr{F}),\Hom^{i-2}_{\widetilde{X}}(\mathscr{G}, \mathscr{F})$ are zero for all $i\leq0$.

Applying $\RHom_{\widetilde{X}}(\mathcal{O}_{\widetilde{X}}(-h), - )$ to \eqref{DT1},
we obtain that
\[\RHom_{\widetilde{X}}(\mathcal{O}_{\widetilde{X}}(-h), \mathscr{G}) \simeq \RHom_{\widetilde{X}}(\mathcal{O}_{\widetilde{X}}(-h), \mathcal{O}_E(-h)[-1]) \simeq \mathbb{C}[-1].\]

Secondly, applying $\RHom_{\widetilde{X}}(\mathscr{G},- )$ to \eqref{DT2}
shows that 
\[\RHom_{\widetilde{X}}(\mathscr{G}, \mathscr{F}) 
\simeq \RHom_{\widetilde{X}}(\mathscr{G}, \mathscr{E}) 
= \RHom_{\widetilde{X}}(\mathbf{L}_{\mathcal{O}_E(-h)}\mathcal{O}_{\widetilde{X}}(-k), \mathbf{L}_{\mathcal{O}_E(-h)}\mathcal{O}_E(-k)).\]

By the adjunction between left and right mutation, we see that 
\[\RHom_{\widetilde{X}}(\mathscr{G}, \mathscr{F}) \simeq  \RHom_{\widetilde{X}}(\mathbf{R}_{\mathcal{O}_E(-h)}\mathbf{L}_{\mathcal{O}_E(-h)}\mathcal{O}_{\widetilde{X}}(-k), \mathcal{O}_E(-k)) = \RHom_{\widetilde{X}}(\mathcal{O}_{\widetilde{X}}(-k), \mathcal{O}_E(-k)) = \mathbb{C}.\]

Finally, apply $\RHom_{\widetilde{X}}(\mathcal{O}_{\widetilde{X}}(-h), - )$ to \eqref{DT2}.
This induces a distinguished triangle
\begin{equation*}
\mathbb{C}[1] \longrightarrow \RHom_{\widetilde{X}}(\mathcal{O}_{\widetilde{X}}(-h),\mathscr{F}) \longrightarrow \RHom_{\widetilde{X}}(\mathcal{O}_{\widetilde{X}}(-h),\mathscr{E})
\end{equation*}
To compute the third term, we apply $\RHom_{\widetilde{X}}(\mathcal{O}_{\widetilde{X}}(-h), - )$ to the triangle of left mutation $\eqref{DT3}$ and obtain that $\RHom_{\widetilde{X}}(\mathcal{O}_{\widetilde{X}}(-h),\mathscr{E}) \simeq \mathbb{C}[-1]$.

Therefore, $\Hom^{i}_{\widetilde{X}}(\mathcal{O}_{\widetilde{X}}(-h),\mathscr{F}) =
\begin{cases}
    \mathbb{C}, & \text{when } i=1,-1, \\
    0, & \text{otherwise.}
\end{cases}$.
This proves the assertion.
\end{proof}

From now on, we fix a choice of $(p_0,p_1,p_2)$,
and we denote by $\mathcal{B}'$ and $\mathcal{B}$ the extension closure of the three-term and six-term Ext-exceptional collection in \ref{Ext ec} respectively.
By Lemma \ref{Mac07}, $\mathcal{B}'$ and $\mathcal{B}$ are hearts of bounded t-structures on $\widetilde{\mathcal{D}}'$ and $\widetilde{\mathcal{D}}$ respectively.

However, it is impossible to construct a heart $\mathcal{A}$ of bounded t-structure on $\mathrm{D^b}(X)$ such that $\mathbf{R}\pi_\ast (\mathcal{B}) \subseteq \mathcal{A}$ because $\mathbf{R}\pi_\ast (\mathcal{O}_{\widetilde{X}}(-h)) = \mathbf{R}\pi_\ast(\mathscr{F}[-2])[1]$. 

We will then tilt $\mathcal{B}$ with respect to a torsion pair whose torsion-free part contains $\mathscr{F}[-2]$, and show that this will produce a localization-compatible heart.

By Lemma \ref{Ext ec}, we can construct a stability condition simply by defining the central charge $Z_\mathcal{B}$ to be the group homomorphism $\mathrm{K}(\mathcal{B})\longrightarrow \mathbb{C}$ such that $Z_\mathcal{B}(\mathcal{O}_{\widetilde{X}}(-h)) =z_1, Z_\mathcal{B}(\mathscr{G})=z_2$, $Z_\mathcal{B}(\mathscr{F}[-2])=z_3$, and $Z_{\mathcal{B}}(\mathcal{O}_{\widetilde{{X}}}(iH)[p_i])=w_i (i=0,1,2)$ for $z_i, w_i \in \mathbb{H}$.

If we choose $z_3$ to be a complex number sufficiently close to the positive real line,
and $z_1=z_2=w_i$ to be any other complex number with $\arg(z_1)>\arg(z_3)$,
then the Harder--Narasimhan filtration of this stability condition implies that the following pair of subcategories $(\mathcal{T}_\mathcal{B},\mathcal{F}_\mathcal{B})$ is a torsion pair.
$$\mathcal{F}_\mathcal{B} := [ \mathcal{F}[-2] ] $$
$$\mathcal{T}_\mathcal{B} := \{B \in \mathcal{B} \ | \ \Hom_\mathcal{B}(B,\mathscr{F}[-2]) = 0 \} $$

\begin{lemma}\label{tildeA}
The tilted heart $\widetilde{\mathcal{A}}:= \mathcal{B}^\sharp = [\mathcal{F}_\mathcal{B} [1],\mathcal{T}_\mathcal{B} ]$ is the extension closure of the set
\[\{ \mathscr{F}[-1],\mathscr{E}[-2],\mathscr{G}, \mathcal{O}_{\widetilde{X}}[p_0], \mathcal{O}_{\widetilde{X}}(H)[p_1], \mathcal{O}_{\widetilde{X}}(2H)[p_2]\}.\]
\end{lemma}
\begin{proof}

The tilted heart $\widetilde{\mathcal{A}}$ is also of finite length,
and should also have $6$ simple generators, i.e. $\mathscr{F}[-1]$ and the universal extensions of $\mathcal{O}_{\widetilde{X}}(-h)$, $\mathscr{G}$, and $\mathcal{O}_{\widetilde{X}}(iH)[p_i](i=0,1,2)$ by $\mathscr{F}[-2]$.

We then compute these universal extensions.
First, consider the distinguished triangle
$$  \mathscr{F}[-2] \otimes \Ext^1_{\widetilde{X}}(\mathcal{O}_{\widetilde{X}}(-h),\mathscr{F}[-2])^\vee \longrightarrow \mathcal{E}_1 \longrightarrow \mathcal{O}_{\widetilde{X}}(-h).$$
By the computation above we know that $\Ext^1_{\widetilde{X}}(\mathcal{O}_{\widetilde{X}}(-h),\mathscr{F}[-2]) = \mathbb{C}$.
Note that this is exactly the shift of (\ref{DT2}), 
and hence the universal extension $\mathcal{E}_1 = \mathscr{E}[-2]$.

As for the universal extension $\mathcal{E}_2$ of $\mathscr{G}$ by $\mathscr{F}[-2]$, since $\Ext^1_{\widetilde{X}}(\mathscr{G},\mathscr{F}[-2])=0$,
we see that $\mathcal{E}_2 = \mathscr{G}$.

Finally, $\mathscr{F}[-2]$ lies in $\widetilde{\mathcal{D}}'$,
by semiorthogonality we have $\RHom(\mathcal{O}_{\widetilde{X}}(iH), \mathscr{F}[-2])=0$ for $i=0,1,2$.
This means that the universal extensions are $\mathcal{O}_{\widetilde{X}}(iH)[p_i]$ themselves, which proves the assertion.
\end{proof}

Note that $\mathbf{R}\pi_\ast(\widetilde{\mathcal{A}})$ is generated by $\mathbf{R}\pi_\ast \mathcal{O}_{\widetilde{X}}(-h), \mathbf{R}\pi_\ast \mathcal{O}_{\widetilde{X}}(-k)$,
and $\mathcal{O}_X(iH)[p_i]$.
In general, the image of a heart under an exact functor need not be a heart.
However, in our case, we may take advantage of the following result in \cite{KS24} to see that the heart $\widetilde{\mathcal{A}}$ we choose will descend to a heart on $\mathrm{D^b}(X)$.

\begin{theorem}[{\normalfont\cite[Theorem 5.8]{KS24}}]\label{KS24}
Let $X$ be a variety of dimension $n \geq 2$ over an algebraically closed
field $k$ of characteristic not equal to $2$ with an ordinary double point
$x_0$ and no other singularities. Let
\[
  \pi : \widetilde{X} = \operatorname{Bl}_{x_0}(X) \longrightarrow X
\]
be the blowup of the singular point, let
$\epsilon \colon E  \hookrightarrow \widetilde{X}$ be the embedding of the
exceptional divisor over $x_0$, and let $S$ be a spinor bundle on $E$.
Then the subcategory
$\widetilde{\mathcal{D}}:= \{F \in \mathrm{D^b}(\widetilde{X})|\epsilon^\ast F \in \langle S,\mathcal{O}_{E}\rangle\}$
is admissible in $\mathrm{D^b}(\widetilde{X})$. Moreover,
\begin{enumerate}
    \item the induced functor
    $\mathbf{R}\pi_{*} : \widetilde{\mathcal{D}} \to \mathrm{D^b}(X)$
    is a crepant Verdier localization;
    \item the kernel $\ker(\pi_\ast|_{\widetilde{{\mathcal{D}}}})$ is generated by a single spherical object $K \in \widetilde{\mathcal{D}}$;
    \item when $\dim(X)$ is even, $K=\epsilon_\ast S$ is $2$-spherical,
     and when $\dim(X)$ is odd, $K$ is $3$-spherical and fits into the distinguished triangle $K \longrightarrow \epsilon_\ast S \longrightarrow \epsilon_\ast S'[2]$, where $S'$ is another spinor bundle on $E$.
\end{enumerate}
\end{theorem}

In our case, $S=\mathcal{O}_E(-h)$ and $S':=\mathcal{O}_E(-k)$,
so the distinguished triangle (\ref{DT3}) says that $K=\mathscr{E}[1]$,
and hence $\ker(\pi_\ast|_{\widetilde{{\mathcal{D}}}})=\langle\mathscr{E}\rangle$.

Generally, in order to descend a heart $\mathcal{H}$ to the Verdier localization $\mathcal{C}/\mathcal{N}$, 
we need to show that $H^i_{\mathcal{H}}(A) \in \mathcal{N}$ for all $A \in \mathcal{N}$ and all $i\in\mathbb{Z}$.

In our situation it is much simpler by applying the following general lemma.

\begin{lemma}\label{heartdescent}
Let $\mathcal{C}$ be a triangulated category and $\mathcal{Q}$ be the heart of a bounded $t$-structure on $\mathcal{C}$.
Let $\mathcal{N}\subset \mathcal{C}$ be a thick subcategory, and let $q\colon \mathcal{C}\longrightarrow \mathcal{C}/\mathcal{N}$
be the Verdier localization.

Assume that $\mathcal{N}$ is generated by a single object $A \in\mathcal{C}$ and that $\mathcal{Q}\cap\mathcal{N}$ is a Serre subcategory of $\mathcal{Q}$ (i.e. closed under subobjects, quotients, and extensions).
Then the following are equivalent:
\begin{enumerate}
\item $\mathcal{Q}$ descends to the heart $\mathcal{Q}/(\mathcal{Q}\cap\mathcal{N})$ of a bounded $t$-structure on $\mathcal{C}/\mathcal{N}$.
\item There exists an object $B\in\mathcal{Q}$ such that
$\mathcal{N} = \langle B\rangle$.
\end{enumerate}
In this case, $q\colon \mathcal{Q}\longrightarrow \mathcal{Q}/(\mathcal{Q}\cap\mathcal{N})$ is exact.
\end{lemma}

\begin{proof}
For simplicity, we denote $H^i:=H^i_{\mathcal{Q}}$ in this proof.

\noindent\emph{(ii) $\Rightarrow$ (i).}
Assume there exists $B\in\mathcal{Q}$ such that $
\mathcal{N} = \langle B\rangle$, 
and define $\mathcal{C}_0 $ to be the set 
\[\{ F\in\mathcal{C} \mid H^i(F)\in\mathcal{N}\ \text{for all }i\in\mathbb{Z} \}.\]

We first check that $\mathcal{C}_0$ is thick. 
Since $\mathcal{N}$ is thick,
$\mathcal{C}_0$ is clearly closed under shifts and direct summands.
Now if $F_1\longrightarrow F_2\longrightarrow F_3$ is a distinguished triangle with $F_1,F_2\in\mathcal{C}_0$,
for each $i$ we consider the long exact sequence
\[
H^i(F_1)\xlongrightarrow{f_i} H^i(F_2)\longrightarrow H^i(F_3)\longrightarrow
H^{i+1}(F_1)\xlongrightarrow{f_{i+1}} H^{i+1}(F_2).
\]

Then there is a short exact sequence in the heart $\mathcal{Q}$
\[
0 \longrightarrow \cok(f_i) \longrightarrow H^i(F_3) \longrightarrow \ker(f_{i+1}) \longrightarrow 0.
\]

Note that $\cok(f_i)$ is a quotient of $H^i(F_2)$ and $\ker(f_{i+1})$ is a subobject of $H^{i+1}(F_1)$.
As $\mathcal{N}\cap\mathcal{Q}$ is a Serre subcategory of $\mathcal{Q}$, 
they are both in $\mathcal{N}\cap\mathcal{Q}$,
and hence we
deduce $H^i(F_3)\in\mathcal{N}\cap\mathcal{Q}$ for all $i$, i.e.\ $F_3\in\mathcal{C}_0$.
Thus $\mathcal{C}_0$ is closed under cones.
Now as $\mathcal{C}_0$ is thick and contains $B$, we get $\mathcal{N} = \langle B\rangle \subset \mathcal{C}_0$.

Conversely, let $F \in \mathcal{C}_0$. Then $H^i(F) \in \mathcal{N}$ for all $i$,
and the standard cohomological filtration expresses $F$ as obtained from the objects $H^i(F)[-i]$ by finitely many shifts and cones.
Since $\mathcal{N}$ is thick, we have $F \in \mathcal{N}$,
and hence $\mathcal{C}_0 \subset \mathcal{N}$.

\noindent\emph{(i) $\Rightarrow$ (ii).}
Assume now that the bounded $t$-structure on $\mathcal{C}$ descends to $\mathcal{C}/\mathcal{N}$. 
This is equivalent to the condition that for every $F \in\mathcal{N}$ and every $i\in\mathbb{Z}, H^i(F)\in\mathcal{N}.$

As $\mathcal{N} = \langle A\rangle$ for some $A\in\mathcal{C}$,
we obtain $H^i(A)\in\mathcal{N}\cap \mathcal{Q}$ for all $i$.
We set $B := \oplus_i H^i(A) \in \mathcal{Q}$ and claim that $\mathcal{N} = \langle B\rangle$.
First, each $H^i(A)\in\mathcal{N}$ and $\mathcal{N}$ is thick, so $B\in\mathcal{N}$ and hence $\langle B\rangle \subset \mathcal{N}.$

On the other hand, $A$ admits a finite cohomological filtration whose successive cones are of the form $H^i(A)[-i]$.
Each $H^i(A)[-i]$ is a direct summand of $B[-i]$,
so $H^i(A)[-i]\in\langle B\rangle$. By closure under cones and shifts we deduce $A\in\langle B\rangle$. 
Therefore
\[
\mathcal{N} = \langle A\rangle \subset \langle B\rangle.
\]

Combining the two inclusions gives $\mathcal{N} = \langle B\rangle$,
with $B\in\mathcal{Q}$, as required.
\end{proof}

Combining Theorem \ref{KS24} and Lemma \ref{tildeA},
we see that $\ker(\pi_\ast|_{\widetilde{\mathcal{D}}})\cap\widetilde{\mathcal{A}}=[\mathscr{E}[-2]]$ is a Serre subcategory of $\widetilde{\mathcal{A}}$,
and then by Lemma \ref{heartdescent},
we immediately see that $\widetilde{\mathcal{A}}$ descends to the heart of a bounded t-structure
$\mathcal{A}:=\pi_\ast(\widetilde{\mathcal{A}})=\widetilde{\mathcal{A}}/(\widetilde{\mathcal{A}}\cap \ker(\pi_\ast|_{\widetilde{\mathcal{D}}}))$ on $\mathrm{D^b}(X)$,
which is of finite length with finitely many simple objects.

\begin{lemma}\label{monolifting}
Let $q \colon \mathcal B \longrightarrow \mathcal A$ be an exact functor between abelian categories.
Let $Z_{\mathcal A}$ be a (weak) stability function on $\mathcal A$, and define $Z_{\mathcal B} := Z_{\mathcal A} \circ q$. 

Assume that the following condition holds:
\begin{enumerate}[label=\textup{($\diamond$)}, ref=\textup{($\diamond$)}, leftmargin=1em]
\item\label{ML}
For every $E \in \mathcal B$ and every monomorphism $i \colon A' \hookrightarrow q(E)$ in $\mathcal A$, there exist a monomorphism $j \colon E' \hookrightarrow E$ in $\mathcal B$  such that $q(E') \simeq A'$.
\end{enumerate}
Then for every $Z_{\mathcal B}$-semistable object $E \in \mathcal B$, the object $q(E)$ is either zero or $Z_{\mathcal A}$-semistable. 

In particular, if $q$ is a Serre quotient $\mathcal{B} \longrightarrow \mathcal{B}/\mathcal{S}$, then $q$ preserves semistability in the above sense.
\end{lemma}
\begin{proof}
Let $E \in \mathcal B$ be $Z_{\mathcal B}$-semistable with $q(E)\neq 0$.
Assume the contrary that $q(E)$ is $Z_{\mathcal A}$-unstable.
Then there exists a nonzero proper subobject $i \colon A' \hookrightarrow q(E)$ which destabilizes $q(E)$.

By our assumption, $A'$ corresponds to a nonzero proper subobject $j \colon E' \hookrightarrow E$ in $\mathcal B$ with $q(E') \simeq A'$.
Now since $Z_{\mathcal B} = Z_{\mathcal A} \circ q$,
we have 
$Z_{\mathcal B}(E') = Z_{\mathcal A}(A')$ and $Z_{\mathcal B}(E) = Z_{\mathcal A}(q(E))$,
and hence
\[\mu_{\mathcal B}(E') =\mu_{\mathcal{A}}(A') > \mu_{\mathcal{A}}(q(E))= \mu_{\mathcal B}(E),\] a contradiction.

For the case that $q$ is a Serre quotient, $q$ satisfies the condition \ref{ML} by \cite[Lemma 14.1.9]{Kra22}.
\end{proof}

We can now prove our main result, that is, we will construct a weak stability condition on $\widetilde{\mathcal{D}}$ and a compatible Bridgeland stability condition on $\mathrm{D^b}(X)$.
In particular, $\Stab(X) \neq \varnothing$.

\begin{theorem}
There exist a weak stability condition $\sigma_{\widetilde{\mathcal{D}}} = (Z_{\widetilde{\mathcal{A}}},\widetilde{\mathcal{A}})$ on $\widetilde{\mathcal{D}}$, satisfying the support property with respect to the quotient lattice $\mathrm{K_{num}}(\widetilde{\mathcal{D}})/\ker (\pi_\ast|_{\widetilde{\mathcal{D}}})$,
and a Bridgeland stability condition $\sigma_{\mathrm{D^b}(X)} = (Z_{\mathcal{A}},\mathcal{A})$ on $\mathrm{D^b}(X)$,
such that
\begin{enumerate}
  \item $\mathbf{R}\pi_\ast\colon \widetilde{\mathcal{D}} \longrightarrow \mathrm{D^b}(X)$ is compatible with the central charges, i.e. $Z_{\widetilde{\mathcal{A}}} = Z_{\mathcal{A}} \circ \pi_\ast$;
  \item $\mathbf{R}\pi_\ast\colon \widetilde{\mathcal{A}} \longrightarrow \mathcal{A}$ is an exact functor;
  \item for any $Z_{\widetilde{\mathcal{A}}}$-semistable object $E$, its pushforward $\mathbf{R}\pi_\ast E$ is either $0$ or $Z_{\mathcal{A}}$-semistable.
  In particular, $\pi_\ast$ induces a morphism $\mathcal{M}_{\sigma_{\widetilde{A}}}(v) \longrightarrow \mathcal{M}_{\sigma_\mathcal{A}}(\pi_\ast v)$ between the moduli spaces for any class $v$ with $\pi_\ast v \neq 0$.
\end{enumerate}
\end{theorem}
\begin{proof}
We pick $z, w_i \in \mathbb{H}$ and define a group homomorphism $Z_{\widetilde{\mathcal{A}}} \colon  \mathrm{K_{num}}(\widetilde{\mathcal{D}})\longrightarrow \mathbb{C}$ by setting 
\[Z_{\widetilde{\mathcal{A}}}(\mathscr{E}[-2]) = 0, Z_{\widetilde{\mathcal{A}}}(\mathscr{F}[-1])=Z_{\widetilde{\mathcal{A}}}(\mathscr{G})=z, Z_{\widetilde{\mathcal{A}}}(\mathcal{O}_{\widetilde{X}}(iH)[p_i])=w_i\]
and extend by linearity.

By Lemma \ref{kernel} we obtain that $\ker (\pi_\ast|_{\widetilde{\mathcal{D}}}) = \langle[\mathscr{E}[-2]], [\mathscr{F}[-1]]-[\mathscr{G}] \rangle \subseteq \mathrm{K_{num}}(\widetilde{\mathcal{A}})$,
which is contained in $\ker Z_{\widetilde{\mathcal{A}}}$.
This shows that $Z_{\widetilde{\mathcal{A}}}$ factors through $\mathrm{K_{num}}(\widetilde{\mathcal{D}})/\ker (\pi_\ast|_{\widetilde{\mathcal{D}}})$.

The essential surjectivity of $\mathbf{R}\pi_\ast$ implies the surjectivity of maps on numerical Grothendieck groups, 
and hence we can define a central charge $Z_\mathcal{A}$ on $\mathcal{A}$ by the compatibility $Z_\mathcal{A} \circ \pi_\ast = Z_{\widetilde{\mathcal{A}}}$.
Under this definition, we see that each simple object in $\mathcal{A}$ takes $Z_\mathcal{A}$-value in $\mathbb{H}$ and hence $Z_\mathcal{A}$ is a stability function.

Note that the Harder--Narasimhan properties and support properties for both $\sigma_{\widetilde{\mathcal{D}}}$ and $\sigma_{\mathrm{D^b}(X)}$ are automatic since $\widetilde{\mathcal{A}}$ and its quotient $\mathcal{A}$ are of finite length with finitely many simple objects.

Finally, (iii) follows immediately from Lemma \ref{monolifting} since $\mathcal{A}=\widetilde{\mathcal{A}}/(\widetilde{\mathcal{A}}\cap \ker(\pi_\ast|_{\widetilde{\mathcal{D}}}))$.
\end{proof}

\begin{remark}
Theorem \ref{KS24} in particular implies that $\mathbf{R}\pi_\ast \colon\widetilde{\mathcal{D}}'\longrightarrow\mathcal{K}u(X)$ is also a Verdier localization with $\ker(\pi_\ast|_{\widetilde{{\mathcal{D}}}'})=\langle\mathscr{E}\rangle$.

We can therefore restrict the construction to obtain a weak stability condition on $\widetilde{\mathcal{D}}'$ and a Bridgeland stability condition on $\mathcal{K}u(X)$,
satisfying the same compatibilities as in the theorem.
\end{remark}

%\nocite{*}
\bibliographystyle{alpha}
\bibliography{ref}
\end{document}